\documentclass[reqno]{amsart}

\usepackage{amsmath}
\usepackage{amsthm}
\usepackage{amsfonts}
\usepackage{amssymb}
\usepackage{enumerate}

\newcommand{\indicator}[1]{\ensuremath{\mathbf{1}_{\{#1\}}}}

\numberwithin{equation}{section}

\newcommand{\E}{\mathbb{E}}
\newcommand{\Prob}{\mathbb{P}}
\newcommand{\C}{\mathbb{C}}
\newcommand{\R}{\mathbb{R}}

\newcommand{\var}{\mathbb{V}}
\newcommand{\V}{\mathbb{V}}
\newcommand{\cov}{\mathrm{Cov}}
\newcommand{\Tr}{\mathrm{Tr}}
\newcommand{\tr}{\mathrm{tr}}
\renewcommand{\d}{\partial}
\renewcommand\Re{\operatorname{Re}}
\renewcommand\Im{\operatorname{Im}}

\theoremstyle{plain}
  \newtheorem{theorem}{Theorem}[section]

  \newtheorem{proposition}[theorem]{Proposition}
  
  \newtheorem{lemma}[theorem]{Lemma}

\theoremstyle{definition}
  \newtheorem{definition}[theorem]{Definition}
  
  \newtheorem{remark}[theorem]{Remark}

\begin{document}
\title[Fluctuations of Matrix Entries]{Fluctuations of Matrix Entries of Regular Functions of Sample Covariance Random Matrices}

\author[S. O'Rourke]{Sean O'Rourke}
\address{Department of Mathematics, University of California, Davis, One Shields Avenue, Davis, CA 95616-8633  }
\thanks{S.O'R. has been supported in part by the NSF grants  VIGRE DMS-0636297 and DMS-1007558}
\email{sdorourk@math.ucdavis.edu}

\author[D. Renfrew]{David Renfrew} \thanks{D.R. has been supported in part by the NSF grants VIGRE DMS-0636297, DMS-1007558, and DMS-0905988 }
\address{Department of Mathematics, University of California, Davis, One Shields Avenue, Davis, CA 95616-8633  }
\email{drenfrew@math.ucdavis.edu}

\author[A. Soshnikov]{Alexander Soshnikov}
\address{Department of Mathematics, University of California, Davis, One Shields Avenue, Davis, CA 95616-8633  }
\thanks{A.S. has been supported in part by the NSF grant DMS-1007558}
\email{soshniko@math.ucdavis.edu}

\begin{abstract}
We extend the results \cite{PRS}, \cite{LytP}, \cite{ors} about the fluctuations of the matrix entries of regular functions of Wigner matrices to 
the case of sample covariance random matrices.
\end{abstract}

\maketitle

\section{Introduction and Main Results}

Recently, there have been a number of results concerning matrix entries of functions of random matrices.  That is, for a $N \times N$ random real 
symmetric (Hermitian) matrix, $M_N$, we consider the entries of the matrix $f(M_N)$ where $f$ is a regular test function.  

In \cite{LP}, Lytova and Pastur consider the case where $M_N$ is drawn from the Gaussian Orthogonal Ensemble (GOE) or Gaussian Unitary Ensemble (GUE).
We recall that a GOE matrix is defined as $M_N=\frac{1}{\sqrt{N}} \* (Y_N+Y_N^t),$ where the entries of $Y_N$ are i.i.d. $N(0, \frac{1}{2}\*\sigma^2)$ 
real random variables (see e.g. \cite{AGZ}).  In a similar way, a GUE matrix is defined as $M_N=\frac{1}{\sqrt{N}} \* (Y_N+Y_N^*),$ where 
the entries of $X_N$ are i.i.d. $N(0, \frac{1}{2}\*\sigma^2)$ complex random variables.
It was shown in \cite{LP} that
\begin{equation}
\label{lp1}
	\sqrt{N} \left( f(M_N)_{ij} - 
\E \left[ f(M_N)_{ij} \right] \right) \longrightarrow N \left( 0, \frac{1 + \delta_{ij}}{\beta} \omega^2(f) \right), 
\end{equation}
in the limit when the size of the matrix goes to infinity, 
where $\omega^2(f) = \var(f(\psi))$, $\psi$ is a random variable distrubited according to Wigner semicircle law, and 
$\beta=1$ for the GOE and 
$\beta=2$ for the GUE.   We recall that the Wigner semicircle distribution is supported on the interval
$[-2\*\sigma, 2\*\sigma]$ and its density with respect to the Lebesgue measure is given by
\begin{equation}
\label{polukrug}
\frac{d \mu_{sc}}{dx}(x) = \frac{1}{2 \pi \sigma^2} \sqrt{ 4 \sigma^2 - x^2} \mathbf{1}_{[-2 \sigma , 2 \sigma]}(x).
\end{equation}

In the case where $W_N=\frac{1}{\sqrt{N}}\*A_N,$ and $A_N$ is a symmetric (Hermitian) Wigner matrix (\cite{AGZ}, \cite{BG}) 
with i.i.d. (not necessarily Gaussian) entries up from the diagonal, 
Pizzo, Renfrew, and Soshnikov studied in \cite{PRS} the 
fluctuations of both the diagonal and off-diagonal entries under the condition that the off-diagonal entries of $A_N$ are centered and 
have finite fourth moment, and the diagonal entries of $A_N$ are centered and
have finite second moment.  The variance of the off-diagonal entries, as before, is equal to $\sigma^2.$ 
The test function $f$ has been assumed to be four times continuously differentiable.   
In particular, it is shown in \cite{PRS} that
\begin{equation*}
	\sqrt{N} \left( f(W_N)_{ij} - \E \left[ f(W_N)_{ij} \right] \right)
\end{equation*}
converges in distribution to the sum of two independent random variables: the first (up to scaling) is given by $(A_N)_{ij}$ and the second is a 
Gaussian random variable with mean zero and variance explicitely given in terms of the function $f$. In addition, it was proven in \cite{PRS} that 
the joint distribution of any finite number of normalized matrix entries converges to the product of one-dimensional limiting distributions. 
If the marginal 
distribution of the entries of $W_N$ is Gaussian (so $W_N$ belongs to the GOE (GUE) ensemble), one recovers (\ref{lp1}).

Such results might be considered as an analogue of the E.Borel theorem for the matrix entries of random matrices from the classical compact groups
(see e.g. \cite{DDN}, \cite{Jiang1}, and \cite{Jiang2}).  In addition, 
the results about the fluctuation of the resolvent quadratic form are related to the limiting distribution of the outliers in the spectrum
of finite rank deformations of Wigner matrices (see e.g. \cite{PRS1} and references therein).

Almost simultaneously with \cite{PRS} and using a different set of ideas, Pastur and Lytova \cite{LytP} 
gave another proof of the limiting distribution of the normalized diagonal entries 
\begin{equation*}
	\sqrt{N} \left( f(W_N)_{ii} - 
\E \left[ f(W_N)_{ii} \right] \right), \qquad 1 \leq i \leq N
\end{equation*}
when $W_N=\frac{1}{\sqrt{N}}\*A_N,$ and $A_N$ is a real symmetric Wigner matrix with i.i.d. centered entries up from the diagonal provided 
the cumulant generating function $\log(\E \exp(z\*A_{12}))$ is entire (so, in particular, all moments of the marginal distribution are finite) and
the test function $f$ satisfies
\begin{equation*}
	\int_\R (1 + 2|k|)^3 |\hat{f}(k)| dk < \infty
\end{equation*}
where $\hat{f}$ is the Fourier transform
\begin{equation} \label{eq:fourier}
	\hat{f}(k) = \frac{1}{\sqrt{2 \pi}} \int_\R e^{-ikx} f(x) dx.
\end{equation}

The results of \cite{PRS} and \cite{LytP} are extended in \cite{ors} to the case of a Wigner matrix with non-i.i.d. entries where it was assumed that 
the off-diagonal entries have 
uniformly bounded fourth moments, diagonal entries have uniformly bounded second moments, and certain Lindeberg type conditions for the fourth moments of 
the off-diagonal entries and the second moments of the diagonal entries are satisfied.
The test function $f(x)$ is assumed to satisfy
\begin{equation*}
	\int_\R (1 + 2|k|)^{2\*s} |\hat{f}(k)|^2 dk < \infty,
\end{equation*}
for some $s>3.$

In this paper, we study the fluctuations of matrix entries of a sample covariance random matrix.  
Namely, we consider the case where 
\begin{equation}
\label{triest}
M_N = \frac{1}{N} A_N A_N^\ast, 
\end{equation}
and $A_N$ is an $N \times n$ rectangular matrix with independent entries.  We begin with some definitions. 

\begin{definition} \label{def:C}
Let $A_N = \left( (A_N)_{ij} \right)_{1 \leq i \leq N; 1 \leq j \leq n}$ be an $N \times n$ matrix with complex entries.  
We say the matrix $A_N$ satisfies condition {\bf C1} if
\begin{enumerate}[(i)]
	\item $\{ \Re (A_N)_{ij}, \Im (A_N)_{ij}: 1 \leq i \leq N; 1 \leq j \leq n \}$ is a collection of independent random variables,
	\item each entry $(A_N)_{ij}$ has mean $0$ and variance $\sigma^2$,
	\item each entry satisifies $\E (A_N)_{i,j}^2 = 0$, 
	\item $\sup_{N,i,j} \E |(A_N)_{ij}|^4 = m_4 < \infty$,
	\item the entries satisfy the Lindeberg condition for the fourth moments, that is, for all $\epsilon >0$,
	\begin{equation} \label{eq:lf4}
		\frac{1}{N^2} \sum_{i,j} \E |(A_N)_{ij}|^4 \indicator{ |(A_N)_{ij}| > \epsilon \sqrt{N}} \longrightarrow 0
	\end{equation}
	as $N \rightarrow \infty$.
\end{enumerate}
\end{definition}

\begin{definition}
Let $A_N = \left( (A_N)_{ij} \right)_{1 \leq i \leq N; 1 \leq j \leq n}$ be an $N \times n$ matrix with real entries.  We say the matrix $A_N$ satisifies 
condition {\bf C2} if $\{ (A_N)_{ij}: 1 \leq i \leq N; 1 \leq j \leq n \}$ is a collection of independent real random variables and conditions (ii), 
(iv), and (v) hold from Definition \ref{def:C}.
\end{definition}

We define $X_N := \frac{1}{\sqrt{N}}A_N$ and $M_N := X_N X_N^\ast$.  Throughout this paper, we assume that 
$c_N := n/N \rightarrow c \in (0, \infty)$ as $N \rightarrow \infty$.  

\begin{definition}
Let $B$ be an $N \times N$ self-adjoint matrix with eigenvalues $\lambda_1, \ldots, \lambda_N$.  The empirical spectral density of $B$ is given by 
\begin{equation*}
	\mu_B := \frac{1}{N} \sum_{i =1}^N \delta_{\lambda_i}.
\end{equation*}
\end{definition}

The limiting empirical spectral density of $M_N$ is known as the Marchenko-Pastur Law (see \cite{bai-book}, \cite{mp}).

\begin{theorem}[Marchenko-Pastur] \label{thm:mp}
Suppose that for each $N$, the entries of $A_N$ are independent complex (real) 
random variables with mean $0$ and variance $\sigma^2$.  Assume $n/N \rightarrow 
c \in (0,\infty)$ and for any $\epsilon>0$
\begin{equation} \label{eq:lf}
	\frac{1}{N^2} \sum_{i,j} \E |(A_N)_{ij}|^2 \indicator{ |(A_N)_{ij}| > \epsilon \sqrt{N}} \longrightarrow 0
\end{equation}
as $N \rightarrow \infty$.  Then with probability one, the emperical density $\mu_{M_N}$ tends to the Marchenco-Pastur distribution, $\mu_{\sigma,c}$, 
with ratio index $c$ and scale index $\sigma^2$ where
\begin{equation*}
	\frac{d \mu_{\sigma,c}}{d x}(x) = \left\{  \begin{array}{lr} 
					\frac{1}{2 \pi x \sigma^2} \sqrt{(u_+ - x)(x - u_-)}, & u_- \leq x \leq u_+, \\
					0, & \text{otherwise},
					\end{array} \right.
\end{equation*}
with a point mass at $0$ with weight $(1-c)$ when $c < 1$, and where
\begin{align*}
	u_+ &:= \sigma^2(1 + \sqrt{c})^2, \\
	u_- &:= \sigma^2(1 - \sqrt{c})^2.
\end{align*}
\end{theorem}

\begin{remark}
We note that the Lineberg condition \eqref{eq:lf} is implied by the Lindeberg condition for the fourth moments \eqref{eq:lf4}.  
\end{remark}

Given a probability measure $\mu$ on the real line, its Stieltjes transform is given by
\begin{equation*}
	\int_\R \frac{d \mu(x)}{z - x}, \quad z \in \C \setminus \text{supp}(\mu).
\end{equation*}
For $\Im z \neq 0$, we have the following bound for the Stieltjes transform of any probability measure on $\R$
\begin{equation} \label{STbound}
	\left| \int_\R \frac{d \mu(x)}{z - x} \right| \leq \frac{1}{|\Im(z)|}.
\end{equation}
The Stieltjes transform of $\mu_{\sigma, c}$ is denoted by $g_{\sigma, c}$ and is characterized as the solution of
\begin{equation}
\label{utro11}
	z \sigma^2 g_{\sigma, c}(z) + (\sigma^2(c-1)-z)g_{\sigma, c}(z) + 1 = 0
\end{equation}
that decays to zero as $z \to \infty$.  

The Stieltjes transform of the expectation of the emperical spectral distribution of $M_N$ is given by
\begin{equation*}
	g_N(z) = \E \int_\R \frac{ d \mu_{M_N}(x)}{z - x} = \E\left[ \tr_N (R_N(z)) \right]
\end{equation*}
where $\tr_N := \frac{1}{N} \Tr$ is the normalized trace and $R_N(z) := (zI_N - M_N)^{-1}$ is the resolvent of $M_N$.  If it does not lead to ambiguity, 
we will use the shorthand notation $R_{ij}(z)$ for $(R_N(z))_{ij}, \ 1\leq i,j\leq N.$

For $s\geq 0, $ we consider the space $\mathcal{H}_s$ consisting of the functions $\phi: \R \to \R$ that satisfy
\begin{equation}
\label{Sobolevnorm}
	\| \phi \|_s^2 := \int_\R (1 + 2|k|)^{2s} |\hat{\phi}(k)|^2 dk < \infty.
\end{equation}
We recall that $C^k(X)$ denotes the space of $k$ times continuously differentiable functions on $X \subset \R$ and define the $C^k(X)$ norm  
\begin{equation}
\label{Cknorm}
	\| \phi \|_{C^k(X)} := \max \left( \left| \frac{d^l f(x)}{dx^l}  \right|,~~ x\in X   , ~~0\leq l \leq k  \right).
\end{equation}

We now present our main results.

\begin{theorem} \label{thm:main_real}
Let $A_N$ be a $N \times n$ random matrix with real entries that satisifies condition {\bf C2}.  Let $m$ be a fixed positive integer and assume that for 
$1 \leq i \leq m$
\begin{equation*} 
	m_4(i) := \lim_{N \rightarrow \infty} \frac{1}{n} \sum_{j} \E|A_{ij}|^4  
\end{equation*}
exists and for all $\epsilon > 0$
\begin{equation} \label{eq:lf0.25}
	\frac{1}{N} \sum_{j=1}^n \E |(A_N)_{ij}|^4 \indicator{ |(A_N)_{ij}| > \epsilon N^{1/4}} \longrightarrow 0, \quad 1 \leq i \leq m,
\end{equation}
as $N \to \infty$.  Assume $c_N \to c \in (0,\infty)$ as $N \rightarrow \infty$ and let $f \in \mathcal{H}_s$ for some $s > 3$.  Then we have the following:

\begin{enumerate}[(i)]
\item The normalized matrix entries
\begin{equation*}
	\left\{ \sqrt{N} \left( f(M_N)_{ij} - \E \left[ f(M_N)_{ij} \right] \right) : 1 \leq i \leq j \leq m \right\}
\end{equation*}
are independent in the limit $N \to \infty$.  

\item For $1 \leq i < j \leq m$, 
\begin{equation*}
	\sqrt{N} \left( f(M_N)_{ij} - \E \left[ f(M_N)_{ij} \right] \right) \longrightarrow N \left( 0, \omega^2(f) \right)
\end{equation*}
in distribution as $N \rightarrow \infty,$ where 
\begin{equation} \label{def:omega}
	\omega^2(f) = \var(f(\eta_c)).
\end{equation}
and $\eta_c$ is a Marchenko-Pastur distributed random variable with ratio index $c$ and scale index $\sigma^2$ and

\item For $1 \leq i \leq m$,

\[ \sqrt{N}( f(M_N)_{ii} - \E[f(M_N)_{ii}]) \to  N\left(0, 2\omega^2(f)+ \frac{\kappa_4(i)}{\sigma^4} \rho^2(f)) \right)  \]
in distribution as $N \rightarrow \infty,$ where
\begin{equation} \label{def:rho}
	\rho(f) =  \E\left[f(\eta_c)\frac{\eta_c - c \sigma^2}{\sqrt{c}\sigma^2}\right]
\end{equation}

and
\begin{equation} \label{eq:real_cumulant4}
	\kappa_4(i) := m_4(i) - 3 \sigma^4.
\end{equation}

\end{enumerate}
\end{theorem}

\begin{theorem} \label{thm:main_complex}
Let $A_N$ be a $N \times n$ random matrix with complex entries that satisifies condition {\bf C1}.  Let $m$ be a fixed positive integer and assume that 
for $1 \leq i \leq m$
\begin{equation*} 
	m_4(i) := \lim_{N \rightarrow \infty} \frac{1}{n} \sum_{j} \E|A_{ij}|^4  
\end{equation*}
exists and for all $\epsilon > 0$ \eqref{eq:lf0.25} holds as $N \to \infty$.  Assume $c_N \to c \in (0,\infty)$ as $N \rightarrow \infty$ and let 
$f \in \mathcal{H}_s$ for some $s > 3$.  Then we have the following:

\begin{enumerate}[(i)]
\item The normalized matrix entries
\begin{equation*}
	\left\{ \sqrt{N} \left( f(M_N)_{ij} - \E \left[ f(M_N)_{ij} \right] \right) : 1 \leq i \leq j \leq m \right\}
\end{equation*}
are independent in the limit $N \to \infty$.  

\item For $1 \leq i < j \leq m$, 
\begin{equation*}
	\sqrt{N} \left( f(M_N)_{ij} - \E \left[ f(M_N)_{ij} \right] \right) \longrightarrow N \left( 0, \omega^2(f) \right)
\end{equation*}
in distribution as $N \rightarrow \infty$ where $N \left( 0, \omega^2(f) \right)$ stands for the complex Gaussian random variable with i.i.d. 
real and imaginary parts with variance $ \frac{1}{2}\*\omega^2(f)$ and $\omega(f)$ is defined in \eqref{def:omega}.  

\item For $1 \leq i \leq m$, 
 
\[ \sqrt{N}( f(M_N)_{ii} - \E[f(M_N)_{ii}]) \to  N\left(0, \omega^2(f)+ \frac{\kappa_4(i)}{\sigma^4} \rho^2(f)) \right)  \]
in distribution as $N \rightarrow \infty$ where $\omega(f)$ is defined in \eqref{def:omega}, $\rho(f)$ is defined in \eqref{def:rho}, and 
\begin{equation*} 
	\kappa_4(i) := m_4(i) - 2 \sigma^4.
\end{equation*}
\end{enumerate}

\end{theorem}

\begin{remark}
The limiting distribution of an entry in the sample covariance case is Gaussian and differs from the Wigner case (\cite{LytP}, \cite{PRS}) 
where the limiting distribution is given by a linear combination of an independent Gaussian random variable and 
the corresponding entry of the Wigner matrix.  However, in the 
square case ($c=1$) the limiting distribution of $\sqrt{N} \left( f(M_N)_{ij} - \E \left[ f(M_N)_{ij} \right] \right)$ coincides with the limiting 
distribution of $\sqrt{N} \left( g(W_N)_{ij} - \E \left[ g(W_N)_{ij} \right] \right),$ where $g(x)=f(x^2)$ and $W_N$ is a Wigner random matrix.
This is not surprising since $M_N$ is the $N\times N$ upper-left corner submatrix of the $(N+n)\times (N+n)$ matrix $Z^2_{N,n},$ where the 
$N\times N$ upper-left  and $n\times n$ lower-right corner submatrices of $Z_{N,n}$ are both zero, the $N\times n$ upper-right corner submatrix of
$Z_{N,n}$ is given by $X_N,$ and the $n\times N$ lower-left corner submatrix of $Z_{N,n}$ is given by $X_N^*.$  The limiting spectral distribution of 
$Z_{N,n}$ in the case $n/N\to c=1$ is given by the Wigner semicircle law and the technique of \cite{PRS}, \cite{ors} in the square case can be extended 
without any difficulties to $Z_{N,n}.$
\end{remark}

\begin{remark}
The functions $1$ and $\frac{x - c \sigma^2}{\sqrt{c}\sigma^2}$ are the first two orthonormal polynomials with respect to $\mu_{\sigma,c}(dx)$.
Therefore, by the Bessel inequality, the variance of the limiting Gaussian distribution 
for the diagonal entries is zero if and only if the test function is linear and the marginal 
distribution is Bernoulli. For the off-diagonal entries, it immediately follows from (\ref{def:omega}) that the variance is zero iff the test 
function is constant on the support of the Marchenko-Pastur law.
\end{remark}

\begin{remark}
It follows from  Proposition \ref{expandvarf} and Lemma \ref{lemma:spectral-norm} that if
$f \in C^7(\R)$ for the diagonal entries $i=j$ ($f \in C^6(\R)$ in the off-diagonal case $i\neq j$),  one can replace 
$\E[f(M_N)_{ij}]$ in Theorems \ref{thm:main_real} and \ref{thm:main_complex} by $\delta_{ij}\*\int f(x) d \mu_{\sigma, c_N}(x)$.
Moreover, as shown in  Proposition \ref{expandvarf}, if $\text{supp}(f) \cap R_+$ is compact, where $R^+= [0, \infty)$ and 
$f$ has seven continuous derivatives, then 
$\E[f(M_N)_{ii}]=\int f(x) d \mu_{\sigma, c_N}(x) + O\left(\frac{1}{N}\right).$ If $f $ has six bounded continuous derivatives on $R_+,$
then $\E \left[ f(M_N)_{ij} \right]=O\left(\frac{1}{N}\right), \ i \neq j.$
\end{remark}

We divide the proof of Theorems \ref{thm:main_real} and \ref{thm:main_complex} into several sections.  In Section \ref{sect:truncation}, we apply a 
standard truncation lemma to the matrix entries of $A_N$.  Section \ref{sect:variance} is devoted to computing the expectation and variance of the 
entries of the resolvent, $R_N(z)$, and Section \ref{sect:functional} extends these results to more general functions.  In 
Section \ref{sect:resolvent}, we prove a central limit theorem for entries of $f(M_N)$ where $f(x)$ is a finite linear combination
of the functions $(z-x)^{-1}, \ z \in \C \setminus \R$.  Finally, we extend this result to more general test functions 
$f \in \mathcal{H}_s$ by an approximation argument.  

\section{Truncation and Extremal Eigenvalues} \label{sect:truncation}

We note that by \eqref{eq:lf4}, we can choose a sequence $\epsilon_N \rightarrow 0$ such that
\begin{equation} \label{eq:lf4-e}
	\frac{1}{\epsilon_N^4 N^2} \sum_{i,j} \E |(A_N)_{ij}|^4 \indicator{ |(A_N)_{ij}| > \epsilon_N \sqrt{N}} \longrightarrow 0
\end{equation}
as $N \rightarrow \infty$.  

\begin{lemma} \label{lemma:truncation}
Assume that $A_N$ is an $N \times n$ matrix that satisifies condition {\bf C1} in the complex case (condition {\bf C2} in the real case).  
Then there exists a random $N \times n$ matrix $\tilde{A}_N$ with independent entries 
and a sequence $\epsilon_N$ which tends to zero as $N$ tends to infinity such that
\begin{enumerate}[(i)]
	\item the entries $(\tilde{A}_N)_{ij}$ have mean zero and variance $\sigma^2$,
	\item $\sup_{i,j} |(\tilde{A}_N)_{ij}| \leq \epsilon_N \sqrt{N}$,
	\item $\sup_{N,i,j} \E |(\tilde{A}_N)_{ij}|^4 < \infty$,
	\item $\Prob(A_N \neq \tilde{A}_N) \longrightarrow 0$ as $N \rightarrow \infty$.
\end{enumerate}
\end{lemma}

\begin{proof}
We present the proof in the case where the entries of $A_N$ are real.  The complex case follows a similar argument.  We begin by selecting a sequence 
$\epsilon_N \rightarrow 0$ such that \eqref{eq:lf4-e} holds.  Then let
\begin{equation*}
	(\hat{A}_N)_{ij} = (A_N)_{ij} \indicator{ |(A_N)_{ij}| \leq \epsilon_N \sqrt{N}}.
\end{equation*}
Define
\begin{align*}
	m_{Nij} &= \E (\hat{A}_N)_{ij} \\
	v^2_{Nij} &= \sigma^2 - \E (\hat{A}_N)^2_{ij}.
\end{align*}
Then we have that
\begin{equation} \label{eq:m-bound}
	| m_{Nij}| \leq \E |(A_N)_{ij}| \indicator{ |(A_N)_{ij}| > \epsilon_N \sqrt{N}} 
\leq\frac{\E |(A_N)_{ij}|^4 \indicator{ |(A_N)_{ij}| > \epsilon_N \sqrt{N}}}{\epsilon_N^3 N^{3/2}} = O\left(\frac{1}{\epsilon_N^{3} N^{3/2}} \right)
\end{equation}
and similarly
\begin{align} \label{eq:v-bound}
& 	v^2_{Nij} \leq \E |(A_N)_{ij}|^2 \indicator{ |(A_N)_{ij}| > \epsilon_N \sqrt{N}} \leq O\left( \frac{1}{\epsilon_N^2 N} \right),\\
\label{eq:mv-bound}
& | m_{Nij}| \leq \frac{v^2_{Nij}}{\epsilon_N N^{1/2}}.
\end{align}

We now define $(\tilde{A}_N)_{ij}$ to be a mixture of 
\begin{enumerate}[(1)]
	\item $(\hat{A}_N)_{ij}$ with probability $1 - \frac{|m_{Nij}|}{\epsilon_N \sqrt{N}} - \frac{v^2_{Nij}}{\epsilon_N^2 N}$; and
	\item a Bernoulli random variable $\xi_{Nij}$ with probability $\frac{|m_{Nij}|}{\epsilon_N \sqrt{N}} + \frac{v^2_{Nij}}{\epsilon_N^2 N}$
\end{enumerate}
where we denote the mean  and the second moment  of $\xi_{Nij}$ by $\mu_{Nij}$ and $\tau^2_{Nij}.$

We can choose $\xi_{Nij}$ such that
\begin{enumerate}[(i)]
	\item $\E(\tilde{A}_N)_{ij} = 0$,
	\item $\var(\tilde{A}_N)_{ij} = \sigma^2$,
	\item $(\tilde{A}_N)_{ij} \leq C \epsilon_N \sqrt{N}$ for some absolute constant $C$.
\end{enumerate}

We now verify that such a construction is possible.  Essentially, we have to show that one can choose $\xi_{Nij}$ in such a way that
(i) and (ii) are satisfied and  
\begin{align} \label{eq:moment-bounds}
	|\mu_{Nij}| \leq C_1 \epsilon_N \sqrt{N} \text{ and } \tau_{Nij}^2  \leq C_2 \epsilon_N^2 N 
\end{align}
for some absolute constants $C_1,C_2>0.$ Indeed, if this is the case, we can construct $\xi_{Nij} = \mu_{Nij} + \psi_{Nij}$ where 
$\psi_{Nij}$ is a symmetric Bernoulli random variable satisfying $|\psi_{Nij}| \leq C \epsilon_N \sqrt{N}$ where $C$ is an absolute 
constant that depends on $C_1$ and $C_2.$ This would immediately follow from
\eqref{eq:moment-bounds}. To verify \eqref{eq:moment-bounds}, we note that
\begin{align*}
	0 &= m_{Nij} \left( 1 - \frac{|m_{Nij}|}{\epsilon_N \sqrt{N}} - \frac{v^2_{Nij}}{\epsilon_N^2 N} \right) + 
\mu_{Nij} \left( \frac{|m_{Nij}|}{\epsilon_N \sqrt{N}} + \frac{v^2_{Nij}}{\epsilon_N^2 N} \right) \\
	\sigma^2 &= (\sigma^2 - v^2_{Nij}) \left( 1 - \frac{|m_{Nij}|}{\epsilon_N \sqrt{N}} - \frac{v^2_{Nij}}{\epsilon_N^2 N} \right) + 
\tau^2_{Nij} \left( \frac{|m_{Nij}|}{\epsilon_N \sqrt{N}} + \frac{v^2_{Nij}}{\epsilon_N^2 N} \right).
\end{align*}
Solving for $\mu_{Nij}$ in the first equation and $\tau_{Nij}$ in the second and applying \eqref{eq:m-bound}-\eqref{eq:mv-bound} yields the 
required bounds \eqref{eq:moment-bounds}, verifying the claim.  

We note that without loss of generality we may assume $C=1$ by our choice of the sequence $\epsilon_N$.  

Next by \eqref{eq:m-bound} and \eqref{eq:v-bound}, we have that 
\begin{equation*}
	\E |(\tilde{A}_N)_{ij}|^4 \leq \E|(A_N)_{ij}|^4 + \left( \frac{|m_{Nij}|}{\epsilon_N \sqrt{N}} + \frac{v^2_{Nij}}{\epsilon_N^2 N} \right) 
(\epsilon_N \sqrt{N})^4 \leq m_4 + O(1).
\end{equation*}

To complete the proof of Lemma \ref{lemma:truncation}, we apply \eqref{eq:m-bound}, \eqref{eq:v-bound}, and \eqref{eq:lf4-e} to obtain  
\begin{align*}
	\Prob(\tilde{A}_N \neq A_N) &\leq \sum_{i,j} \left( \frac{|m_{Nij}|}{\epsilon_N \sqrt{N}} + \frac{v^2_{Nij}}{\epsilon_N^2 N} + 
\Prob( |(A_N)_{ij}| > \epsilon_N \sqrt{N}) \right) \\
		& \leq \frac{3}{\epsilon_N^4 N^2} \sum_{i,j} \E |(A_N)_{ij}|^4 \indicator{|(A_N)_{ij}| > \epsilon_N \sqrt{N}} \longrightarrow 0
\end{align*}
as $N \rightarrow \infty$.  
\end{proof}

We can now apply Lemma \ref{lemma:truncation} to obtain a result on the norm of the matrix $\frac{1}{N} A_N A_N^\ast$.  This result follows from 
\cite[Theorem 5.9]{bai-book}.  We present it here for completeness.  

\begin{lemma} \label{lemma:spectral-norm}
Under the assumptions of Lemma \ref{lemma:truncation}, we have that $\|\frac{1}{N}A_N A_N^\ast \| \longrightarrow \sigma^2 (1+\sqrt{c})^2$ 
in probability as $N \rightarrow \infty$.  
\end{lemma}
\begin{proof}
Since Theorem 5.9 from \cite{bai-book} does not apply directly to $\frac{1}{N}A_N A_N^\ast$, we simply note that by Lemma \ref{lemma:truncation}, it is 
enough to show $\|\frac{1}{N} \tilde{A}_N \tilde{A}_N^\ast \| \longrightarrow \sigma^2 (1+\sqrt{c})^2$ in probability.  Theorem 5.9 from \cite{bai-book} 
now applies to the matrix $\frac{1}{N}\tilde{A}_N \tilde{A}_N^\ast$ to obtain
\begin{equation*}
	\Prob\left( \left\| \frac{1}{N}\tilde{A}_N \tilde{A}_N^\ast \right\| > \sigma^2 (1 + \sqrt{c})^2 + x\right) \longrightarrow 0
\end{equation*}
as $N \rightarrow \infty$ for all $x > 0$.  The proof is then complete by noting that Theorem \ref{thm:mp} implies that, with probability $1$,
\begin{equation*}
	\limsup_{N \rightarrow \infty} \left\| \frac{1}{N}\tilde{A}_N \tilde{A}_N^\ast \right\| \geq \sigma^2 (1+ \sqrt{c})^2.
\end{equation*}
\end{proof}

We also note that by \eqref{eq:lf0.25}, we can choose a sequence $\epsilon_N \rightarrow 0$ such that
\begin{equation} \label{eq:lf0.25-e}
	\frac{1}{\epsilon_N^4 N} \sum_{j=1}^n \E |(A_N)_{ij}|^4 \indicator{ |(A_N)_{ij}| > \epsilon_N N^{1/4}} \longrightarrow 0
\end{equation}
as $N \rightarrow \infty$ for $1 \leq i \leq m$. 

\begin{lemma} \label{lemma:finite-truncation}
Let $A_N$ be a $N \times n$ complex (real) matrix that satisifies condition {\bf C1} ({\bf C2}) 
and \eqref{eq:lf0.25-e} for $1 \leq i \leq m$, where $m$ is a fixed positive 
integer.  Then there exists a random $N \times n$ matrix $\tilde{A}_N$ with independent entries 
and a sequence $\epsilon_N$ which tends to zero as $N$ tends to infinity such that
\begin{enumerate}[(i)]
	\item $(\tilde{A}_N)_{ij} = (A_N)_{ij}$ for $m < i \leq N$ and $1 \leq j \leq n$,
	\item the entries $(\tilde{A}_N)_{ij}$ have mean zero and variance $\sigma^2$,
	\item $\sup_{j,1 \leq i \leq m} |(\tilde{A}_N)_{ij}| \leq \epsilon_N N^{1/4}$,
	\item $\sup_{N,i,j} \E |(\tilde{A}_N)_{ij}|^4 < \infty$,
	\item $\Prob(A_N \neq \tilde{A}_N) \longrightarrow 0$ as $N \rightarrow \infty,$
        \item $\frac{1}{n} \sum_{j} ( \E|A_{ij}|^4 -\E|\tilde{A}_{ij}|^4) \to 0, \ 1\leq i\leq m.$
\end{enumerate}
\end{lemma}

The proof of Lemma \ref{lemma:finite-truncation} is very similar to the proof of Lemma \ref{lemma:truncation} and the details are left to the reader.  

By Lemmas \ref{lemma:truncation} and \ref{lemma:finite-truncation}, we will assume that all the entries of $A_N$ are bounded by $\epsilon_N \sqrt{N}$ 
and that the entries satisify conditions (ii)-(iv) of Lemma \ref{lemma:finite-truncation} for the remainder of the paper.  Indeed, since the truncated 
matrix coincides with the original with probability going to $1$, it is enough for us to prove Theorems \ref{thm:main_real} and \ref{thm:main_complex} 
for the truncated matrix.  

We will also need the following lemma for controlling the expectation of the norm of $X_N X_N^\ast$.  

\begin{lemma} \label{lemma:expecation-norm}
For any $k \geq 1$, there exists a constant $C>0$ (depending only on $\sigma$, $c$, and $k$) such that 
\begin{equation*}
	\E \left[ \| X_N X_N^\ast \|^k \right] \leq C
\end{equation*}
for $N$ sufficiently large.  
\end{lemma}

\begin{proof}
For any $\epsilon>0$, 
\begin{align*}
	\E \left[ \| X_N X_N^\ast \|^k \right] &= \E \left[ \|X_N X_N^\ast \|^k \indicator{ \|X_N X_N^\ast\| \leq \sigma^2(1+\sqrt{c})^2 + \epsilon} 
\right] \\
	& \qquad + \E \left[ \|X_N X_N^\ast \|^k \indicator{ \|X_N X_N^\ast\| > \sigma^2(1+\sqrt{c})^2 + \epsilon} \right] \\
	& \leq \left[ \sigma^2(1+\sqrt{c})^2 + \epsilon \right] ^ k + k \int_\epsilon^\infty t^{k-1} 
\Prob\left( \| X_N X_N^\ast \| > t + \sigma^2(1 + \sqrt{C})^2 \right) dt.
\end{align*}
By \cite[Theorem 5.9]{bai-book}, we have that
\begin{align*}
	\int_\epsilon^\infty &t^{k-1} \Prob\left( \| X_N X_N^\ast \| > t + \sigma^2(1 + \sqrt{C})^2 \right) dt \\
	&\qquad \leq C' N^{-k -2} \int_\epsilon^\infty t^{k-1} \left( \sigma^2(1+\sqrt{C})^2 + t - \epsilon \right)^{-k -2}dt = O(N^{-k -2})
\end{align*}
for some constant $C'>0$.  Thus,  
\begin{equation*}
	\E \left[ \| X_N X_N^\ast \|^k \right] \leq \left[ \sigma^2(1+\sqrt{c})^2 + \epsilon \right]^k + 1
\end{equation*}
for $N$ sufficiently large.  
\end{proof}

\section{ Mathematical Expectation and Variance of Resolvent Entries} \label{sect:variance}
\label{sec:estvar}
This section is devoted to the estimates of the mathematical expectation and the variance of the resolvent entries.  
Throughout the section, we will consider the real case.  The proofs in the complex case are very similar. 
It follows from Lemmas \ref{lemma:truncation} and \ref{lemma:finite-truncation} that for the purposes of the proof of Theorems 
\ref{thm:main_real} and \ref{thm:main_complex} we can assume that $A_N$ satisfies properties (i)-(iii) in Lemma \ref{lemma:truncation} and properties
(ii)-(iv) in Lemma \ref{lemma:finite-truncation}.  Indeed, such a truncated matrix coincides with $A_N$ with probability going to $1,$ and, therefore, 
if the results of Theorems \ref{thm:main_real} and \ref{thm:main_complex} hold for the truncated matrix, they also hold for $A_N.$

We begin by recalling the 
basic resolvent identity
\begin{equation}
\label{resident}
	(z I - A_2)^{-1} = (z I - A_1)^{-1} - (z I - A_1)^{-1} (A_1 -A_2) (zI - A_2)^{-1} 
\end{equation}
which holds for all $z \in \C$ where $(z I - A_1)$ and $(z I - A_2)$ are invertible.  

We will also use the decoupling formula (see for example \cite{KKP} and \cite{LytPastur}): 
for any real-valued random variable, $\xi$, with $p+2$ finite moments and $\phi$ a 
complex-valued function with $p+1$ continuous and bounded derivatives the decoupling formula is given by:
\begin{equation}
\label{decouple}
	\E(\xi \varphi(\xi)) = \sum_{a=0}^p \frac{\kappa_{a+1}}{a!} \E(\varphi^{(a)}(\xi)) + \epsilon 
\end{equation}
where $\kappa_a$ are the cumulants of $\xi$ and $\epsilon \leq C \sup_t | \varphi^{(p+1)}(t)| \E(|\xi|^{p+2})$, $C$ depends only on $p$. 
It follows from the proof of the decoupling formula in \cite{LytPastur} that
if $ |\xi|\leq K$ with probability $1,$ then the supremum in the upper bound for the error term can be taken over $t \in [-K, K].$

Recall that we denote the entries of the resolvent $R_N(z)=(zI_N - M_N)^{-1}$ of $M_N$ by $R_{ij}(z), \ 1\leq i,j\leq N.$
Using \eqref{resident}, we can compute the derivatives of the resolvent with respect to any entry
\begin{equation} \label{eq:derivative}
\begin{split}
\frac{\d R_{ij} }{\d X_{kl}} &= R_{ik} (X^* R)_{lj} + (RX)_{il} R_{kj}.
\end{split}
\end{equation} 

We now use \eqref{decouple} and \eqref{eq:derivative} to compute the expectation and variance of the resolvent entries.  

\begin{proposition}
\label{expandvar}
Let $M_N =\frac{1}{N} A_N A_N^* $ be a random real (complex) sample covariance matrix satisfying condition 
\textbf{C2} (\textbf{C1})  and $R_N(z) = (zI_N - M_N)^{-1}$.  Then 
\begin{equation}
\begin{split}
\label{expii}
&\E[R_{ii}(z)] = g_{\sigma,c_N}(z) + O\left(\frac{P_6(|\Im(z)|^{-1})}{ N}\right)  \\
  & \text{ for } 1 \leq i \leq N, \text{ uniformly on bounded sets of } \C \setminus\R
\end{split}
\end{equation}
\begin{equation}
\begin{split}
\label{expik}
&\E[R_{ik}(z)] =  O\left(\frac{P_5(|\Im(z)|^{-1})}{ N}\right) \\
&\text{ for } 1 \leq i \not= k \leq N, \text{ uniformly on } z \in \C\setminus\R
\end{split}
\end{equation}
\begin{equation}
\begin{split}
\label{varij}
 &\V[R_{ij}(z)] = O\left(\frac{P_4(|\Im(z)|^{-1})  \E \left[ P_{10}(\|X_N\|) \*( \|R_N(z)\|^2 + \| R_N(z) \|^{3/2}) \right]}{ N}\right)  \\
 & \text{ for } 1 \leq i , j \leq N, \text{ uniformly on } z \in \C \setminus \R.
\end{split}
\end{equation}
\end{proposition}
Here and throughout the paper $P_k$ denotes a polynomial of degree $k$ with non-negative coefficients. 

In \eqref{varij} we have included the norm of the resolvent in the error estimate. This will be useful in the proof of Proposition \ref{varijf}.

\begin{proof}
The following inequalities will be useful in our calculations:
\begin{equation}
\label{ineq}
|R_{ik}(z)| \leq |\Im(z)|^{-1},   \quad   \sum_{j=1}^N |R_{ij}(z)|^2 \leq \|R_N(z)\|^2 \leq |\Im(z)|^{-2} 
\end{equation}
We first prove \eqref{expii} and \eqref{expik}.
We define the following sets on the complex plane. Let $T$ be an arbitrary large number. Let $L$ be a sufficiently large constant, to be chosen later.

\begin{equation}
\label{QandO}
\mathcal{Q}_N := \{ z : |z|<T+1 \text{ and } |\Im(z)| > L N^{-1/5} \} \quad \mathcal{O}_N := \{ z : |\Im(z)| > L N^{-1/4} \}
\end{equation}

Note that if $z \in \mathcal{Q}_N^c \cap \{|z| < T+1\}$ then $|\Im(z)|^{5} \leq L^5 N^{-1}$. When combined with \eqref{STbound} this implies 
\begin{equation}
\label{QSTbound}
| \E[R_{ii}(z)] - g_{\sigma,c_N}(z) | \leq 2|\Im(z)|^{-1} = O\left(\frac{|\Im(z)|^{-6}}{N} \right).
\end{equation}
Similarly, if $z \in \mathcal{O}_N^c$ then   
\begin{equation}
\label{OSTbound}
| \E[R_{ik}(z)]  | \leq |\Im(z)|^{-1} = O\left(\frac{|\Im(z)|^{-5}}{N} \right).
\end{equation}

For the remainder of the proof of \eqref{expii} and \eqref{expik} we will assume that $z$ is in $\mathcal{Q_N}$ and $\mathcal{O_N}$, respectively. 

The proof of both statements begins with the resolvent identity \eqref{resident}, and then an application of the decoupling formula \eqref{decouple}.
\begin{equation}
\label{master}
\begin{split}
& z\E[R_{ik}(z)] 
= \delta_{ik} +  \sum_{j=1}^N \sum_{l=1}^n \E[ R_{ij}(z) X_{jl}  X_{kl} ] 
\\
& = \delta_{ik} + \sigma^2 \E[R_{ik}(z) \tr_N( R_N(z) X X^*)] + \frac{\sigma^2}{N} \E[ (R_N(z)X X^*R(z))_{ik}  ] + 
\frac{n\sigma^2}{N} \E[ R_{ik}(z) ] + r_N,
\end{split}
\end{equation}
where $r_N$ is the third cumulant term coming from $p=2$ and the error from truncating at $p=2$.

From the definition of the resolvent we have $R_N(z)(z I_N - XX^*) = I_N$, which implies $R_N(z)XX^* = zR_N(z) - I_N$. Applying this identity to 
\eqref{master} yields:
\begin{equation}
\label{master1}
\begin{split}
z\E[R_{ik}(z)] =&  \delta_{ik} +  \sigma^2z \E[R_{ik}(z) \tr_N( R_N(z)) ] - \sigma^2\E[R_{ik}(z)]  +  \\
& ~~~ \frac{\sigma^2}{N} \E[ (R_N(z)X X^*R_N(z))_{ik}  ] + \frac{n\sigma^2}{N} \E[ R_{ik}(z)] + r_N. 
\end{split} 
\end{equation}
We begin with the following lemma:
\begin{lemma}
\label{error1}
For $z \in \C \setminus \R$:
\begin{align}
\label{houston}
& Cov[R_{ij}(z), \tr_N(R_N(z)) ] \leq \frac{ P_{2}(|\Im(z)|^{-1})\E[\|R_N(z)\|^{3/2}]}{N},\\
\label{atlantaz}
& r_N \leq  \frac{P_{4}(|\Im(z)|^{-1})}{N}.
\end{align}
Additionally, for $z \in \mathcal{O}_N$:
\begin{equation}
\label{atlantaR}
r_N \leq  \frac{P_{2}(|\Im(z)|^{-1})\E[ P_8(\|X_N\|) \*( \|R_N(z)\|^2 + \| R_N(z) \|^{3/2})}{N}.
\end{equation}
\end{lemma}

\begin{proof}
To prove \eqref{houston} we begin with the following bounds from Proposition 4 in \cite{Sh}:
\begin{equation}
\label{tracevarbound}
 \V(\tr_N(R_N(z))) \leq \frac{|\Im(z)|^{-4}}{N^2}, \quad  \V(\tr_N(R_N(z))) \leq \frac{|\Im(z)|^{-7/2} \E[\|R_N(z)\|]^{3/2}}{N^{2}}.
\end{equation}
It follows from the proof of Proposition 4 in \cite{Sh} that these bounds are valid provided the fourth moments are uniformly bounded (\cite{Sh1}).
Additionally, from \eqref{ineq} we have
\begin{equation}
 \V(R_{ij}(z)) \leq |\Im(z)|^{-1/2} \E[\|R_N(z)\|]^{3/2}.
\end{equation}
Using Cauchy-Schwarz this implies
\begin{equation}
 Cov(R_{ij}(z), \tr_N(R_N(z))) \leq \frac{P_2(|\Im(z)|^{-1}) \E[\|R_N(z)\|^{3/2}]}{N}
\end{equation}
as desired.

Now we prove \eqref{atlantaR}; the argument along with Lemma \ref{lemma:expecation-norm} 
can be modified to prove \eqref{atlantaz}. The third cumulant term in the decoupling formula is:
\begin{equation*}
\begin{split}
& \frac{1}{2 N^{3/2}} \sum_{j=1}^N \sum_{l=1}^n \kappa_3((A_N)_{kl}) \E\left[\frac{\d^2 R_{ij}(z) X_{jl}}{\d X_{kl}^2} \right] \\
& =\sum_{j=1}^N \sum_{l=1}^n \kappa_3((A_N)_{kl}) [R_{ik}(z) (X^* R(z)X)_{ll} R_{kj}(z) X_{jl} \\
& =  \frac{1}{N^{3/2}} \sum_{l=1}^n \kappa_3((A_N)_{kl})\E[R_{ik}(z) (X^* R_N(z)X)_{ll} (R_N(z)X)_{kl} \\
& ~~~+ (R_N(z)X)_{il} (R_N(z)X)_{kl} (R_N(z)X)_{kl}  + R_{ik}(z) (R_N(z)X)_{kl} \\
& ~~~ + R_{ik}(z) (X^* R_N(z))_{lk} (X^* R_N(z) X)_{ll}  + (R_N(z)X)_{il} R_{kk}(z)(X^*R_N(z) X)_{ll}],
 \end{split}
\end{equation*}
where $\kappa_3((A_N)_{kl})$ is the third cumulant of $(A_N)_{kl}$. By condition \textbf{C2}, the $\kappa_3((A_N)_{kl})$'s are uniformly bounded.\\
Using \eqref{ineq} and the Cauchy-Schwarz inequality this term is seen to be \
$$O\left(\frac{|\Im(z)|^{-1} \E[P_4(\|X_N\|)\*\|R_N(z)\|^2]}{N}\right).$$

The truncation error is bounded from above by a finitely many sums of the following form
\begin{equation}
\label{dinamo}
\frac{C \* m_4}{N^2} \* \sum_{l=1}^n  \sup \E |R'_{ab}(z) (R_N'(z) X')_{cd}^{\alpha}(R_N' (z) X')_{ef}^{\beta} (X'^{*} R_N'(z) )_{gh}^{\gamma} (X'^{*} 
R_N' X')_{qr}^{\delta} |,
\end{equation}
where the $\sup$ is over all rank two perturbations of $X$ of the form $X' = X + x E_{kl}$ where $(E_{kl})_{ij} = \delta_{ik}\delta_{jl}+
\delta_{il}\delta_{ij}$ and $R_N'(z) = (z I_N - X' X'^{*})^{-1}$. Additionally, $\alpha + \beta + \gamma +2\delta\leq 4$ and each of 
$a,b,c,d,e,f,g,h,q,r $ are one of $i,l,k$.
The bound (\ref{atlantaz}) then immediately follows from (\ref{ineq}) and Lemma \ref{lemma:expecation-norm}.

To prove the bound (\ref{atlantaR}), we can assume by (iii) of 
Lemma \ref{lemma:finite-truncation} that $|x| \leq \epsilon_N N^{-1/4}$.  Then
\begin{equation} 
\label{dinamo1}
\|X_N'\| = \|X_N\| + o(1).
\end{equation}
 Additionally, 
\begin{equation*}
R_N'(z)= R_N(z) +R_N(z) \* (xE_{kl}X_N^* + x \*X_N\* E_{lk} + x^2 E_{kl}E_{lk}  )\* R_N'(z). 
\end{equation*}
Thus,
\begin{equation}
\label{dinamo2}
\|R_N'(z)\|\leq \|R_N(z)\| \* (1 + |x| \*\frac{1}{|\Im z|}\* \|X_N\|) \leq \|R_N(z)\| \* (1 + \epsilon_N \* \|X_N\|).
\end{equation}
Using (\ref{dinamo}-\ref{dinamo2}), one obtains (\ref{atlantaR}).
\end{proof}

It follows from \eqref{master1} and Lemma \ref{error1} that for $i = k$
\begin{equation}
\begin{split}
\label{masterdiag}
(z - \sigma^2 z g_N(z) - \sigma^2 c_N + \sigma^2) \*\E R_{ii}(z) &= 1 + O\left(\frac{P_4(|\Im(z)|^{-1})}{N}\right).\\
\end{split} 
\end{equation}
Summing over $i$ and dividing by $N$ gives
\begin{equation}
\begin{split}
\label{masterdiag1}
(z - \sigma^2 z g_N(z) - \sigma^2 c_N + \sigma^2) g_N(z) &= 1 + O\left(\frac{P_4(|\Im(z)|^{-1})}{N}\right).
\end{split} 
\end{equation}
Additionally if $i \not= k$
\begin{equation}
\begin{split}
\label{masteroffdiag}
(z - \sigma^2 z g_N(z)- \sigma^2 c_N + \sigma^2)  \E[  R_{ik}(z)] &= O\left(\frac{P_4(|\Im(z)|^{-1})}{N}\right)
\end{split} 
\end{equation}
Now we use \eqref{masterdiag} and \eqref{masterdiag1} and the following lemma to complete the proof of \eqref{expii}. 
\begin{lemma}
\label{lowerb}
On $\mathcal{Q}_N$, $|g_N(z) - \sigma^2  g_N^2(z)| $ is uniformly bounded in $z$ and $N$ from below by a positive constant. 
\end{lemma}
\begin{proof}
Assume it is not, then for any $\delta > 0$ there would exist a $z$ such that $|g_N(z) - \sigma^2  g_N^2(z) |< \delta^2$ 
which in turn implies $|g_N(z)| < \delta$ or $|\sigma^2 g_N(z) -1| < \delta$. But $|g_N(z)| < \delta$ contradicts \eqref{masterdiag1} 
if $\delta$ is sufficiently small once
$L$ from \eqref{QandO} is chosen to make $ O\left(\frac{P_4(|\Im(z)|^{-1})}{N}\right)$ small enough because 
$(z - \sigma^2 z g_N(z) - \sigma^2 c_N + \sigma^2) $ is bounded on $\mathcal{Q}_N$.

On the other hand if $|\sigma^2 g_N(z) - 1| < \delta$ then 
\[ |(z - \sigma^2 z g_N(z) - \sigma^2 c_N + \sigma^2) g_N(z) - 1| \geq \Big|  |z g_N(z) - \sigma^2 z g^2_N(z)| -  | \sigma^2 (1-c_N ) g_N(z) - 1|    
\Big|.\]
but
\[|z g_N(z) - \sigma^2 z g^2_N(z)| < (T+1)\delta \]
\[ |(\sigma^2 (1-c_N) g_N(z) - 1)+ c_N| = |(1-c_N)(\sigma^2 g_N(z) -1)| \leq |(1-c_N) \delta| \]
So for $\delta$ small and $L$ sufficiently large we reach a contradiction with (\ref{masterdiag1}).
\end{proof}

Let $s_N(z) := \frac{1 + \sigma^2 g_N(z) ( c_N -1)}{g_N(z) - \sigma^2  g_N^2(z) }$ then by Lemma \ref{lowerb} and \eqref{masterdiag} we have:
\begin{equation}
\label{approxeq}
\begin{split}
s_N(z) - z =  \frac{1 + \sigma^2 g_N(z) ( c_N -1)}{g_N(z) - \sigma^2  g_N^2(z) } - z =O\left(\frac{P_4(|\Im(z)|^{-1})}{N}\right)
\end{split} 
\end{equation}

Finally, for  $z \in \mathcal{Q}_N$, $g_N(z) = g_{\sigma,c_N}(s_N(z)),$ which can be seen by evaluating
\[ z \sigma^2  g_{\sigma,c_N}^2(z) +(\sigma^2(c_N-1) -z)g_{\sigma,c_N}(z) + 1 = 0\]
at $s_N(z)$. This yields:
\[ \frac{1 + \sigma^2 g_N(z)  (c_N -1)}{g_N(z) - \sigma^2  g_N^2(z) } \left(\sigma^2  g_{\sigma,c_N}^2(s_N(z)) - g_{\sigma,c_N} (s_N) \right) +
\sigma^2(c_N-1) g_\sigma^2(s_N(z)) +1=0 \]
Rearranging this equation and applying the estimates in \eqref{ineq} gives:
\begin{equation}
\label{durham}
\big( g_{\sigma,c_N}(s_N(z))-g_N(z) \big)\left(\frac{2\sigma^2}{|\Im(z)|^{-1}} + \frac{\sigma^4(c_N-1)}{|\Im(z)|^{-2}} \right)  
\geq \big( g_{\sigma,c_N}(s_N(z))-g_N(z) \big).
   \end{equation}
So $g_{\sigma,c_N}(s_N(z))=g_N(z)$ for sufficiently large $z$ and then on $\mathcal{Q_N}$ by analytic continuation.
 
Combining \eqref{durham} and \eqref{approxeq} gives:  
\[ |g_N(z) - g_{\sigma,c_N}(z)| = \left| \int \frac{d \mu_{\sigma,c_N}(x)}{z-x} - \frac{d\mu_{\sigma,c_N}(x)}{s_N(z)-x} \right| \leq  
\left|C (s_N(z) -z) \int \frac{d \mu_{\sigma,c_N}(x)}{(z-x)^2} \right|    \leq  \frac{P_6(|\Im(z)|^{-1})}{N}\]  on $ \mathcal{Q}_N.$
%
This completes the proof of \eqref{expii}.

Beginning from \eqref{masteroffdiag} we now finish the proof of \eqref{expik}. 
We have
\begin{equation}
\label{asdf}
\begin{split}
g_N(z)(z - \sigma^2 z g_N(z)- \sigma^2 c + \sigma^2)  \E[  R_{ik}(z)] &= g_N(z) +O\left(\frac{P_4(|\Im(z)|^{-1})}{N}\right) \\
\left(1 + O\left(\frac{P_4(|\Im(z)|^{-1})}{N}\right)\right) \E[  R_{ik}(z)] &=  O\left(\frac{P_5(|\Im(z)|^{-1})}{N}\right) \\
\end{split} 
\end{equation}

Recall that $z \in \mathcal{O}_N$, and $L$ can be chosen such that $O\left(\frac{P_4(|\Im(z)|^{-1})}{N}\right)$ on the l.h.s of  \eqref{asdf} 
is less than $1/2$ in absolute value. Then:
\[ \E[  R_{ik}(z)] =  O\left(\frac{P_5(|\Im(z)|^{-1})}{N}\right)\]

This completes the proof of \eqref{expik}.

Our final step in the proof of Proposition \ref{expandvar} is to prove \eqref{varij}.
First note that if $z \in \mathcal{O}_N^c$ then
\[ \V[R_{ij}(z)] \leq \E[ \|R_N(z)\|^2] \leq \frac{L^4 \E[\|R_N(z)\|^2]}{|\Im(z)|^{4}N} \]
For the remainder of the proof we will assume $z \in \mathcal{O}_N$.
We begin with the resolvent identity \eqref{resident} applied to $\E[R_{ik}(z) R_{ik}(\overline{z})]$ and then apply the decoupling formula 
(\ref{decouple}):
\begin{align}
& z \E[R_{ik}(z) R_{ik}(\overline{z})] = \delta_{ik} \E[R_{ik}(\overline{z})] + \sum_{j=1}^N \* \sum_{l=1}^n \*
\E[ (R_{ij}(z)\* R_{ik}(\overline{z})\*X_{jl}\*X_{kl}] \\
& =  \delta_{ik}\E[R_{ik}(\overline{z})]  +\sigma^2 \E [R_{ik}(z)R_{ik}(\overline{z})  \tr_N(X^* R_N(z)X)] + \frac{\sigma^2}{N} 
\E[(R_N(z) X X^* R_N(z))_{ik} R_{ik}(\overline{z}) ]  \\
&~~~+  \frac{\sigma^2}{N}\E[( R_{ik}(\overline{z}) (R_N(z) X X^* R_N(\overline{z}))_{ik} + (R_N(\overline{z})X X^* R_N(z))_{ii} 
R_{kk}(\overline{z}) )] + \frac{\sigma^2 n}{N} \E[R_{ik}(z) R_{ik}(\overline{z}) ] + r_N,
\end{align}
where $r_N$ contains the third cumulant term, $p=2$, and the error for truncating at $p=2$.

Once again using that $R_N(z)XX^* = zR_N(z) - I_N$ gives:
\begin{equation}
\begin{split}
\label{mastervar}
z \E[R_{ik}(z) R_{ik}(\overline{z})] &=  \delta_{ik}\E[R_{ik}(\overline{z})]  +\sigma^2 z \E [R_{ik}(z)R_{ik}(\overline{z})  
\tr_N(R_N(z))] - \sigma^2 \E [R_{ik}(z)R_{ik}(\overline{z}) ] \\
&~~~ + \frac{\sigma^2}{N} \E[(R_N(z) X X^* R_N(z))_{ik} R_{ik}(\overline{z}) ]  +\frac{\sigma^2}{N}\E[( R_{ik}(\overline{z}) 
(R_N(z) X X^* R_N(\overline{z}))_{ik} \\
&~~~+ (R_N(\overline{z})X X^* R_N(z))_{ii} R_{kk}(\overline{z}) )] +   \frac{\sigma^2 n}{N} \E[R_{ik}(z) R_{ik}(\overline{z}) ] + r_N
\end{split}
\end{equation}

Similar to Lemma \ref{error1} we use the following lemma to complete our variance bound.
\begin{lemma}
\label{error2}
For $z \in \C \setminus \R$:
\begin{equation}
\label{houston2}
Cov[R_{ij}(z)  R_{ik}(\overline{z})   , \tr_N(R_N(z)) ] \leq \frac{ P_{3}(|\Im(z)|^{-1})\E[\|R_N(z)\|^{3/2}]}{N}
\end{equation}
\begin{equation}
\label{atlanta2}
r_N \leq  \frac{P_{3}(|\Im(z)|^{-1})\E \left[ P_{10}(\|X_N\|) \*( \|R_N(z)\|^2 + \| R_N(z) \|^{3/2}) \right]}{N}
\end{equation}
\end{lemma}

\begin{proof} The proof follows from the steps taken in the proof of Lemma \ref{error1}. For the reader's convenience the third cumulant term is: 
\begin{equation}
\begin{split}
&\frac{1}{2N^{3/2} } \sum_{j=1}^N \sum_{l=1}^n \kappa_3((A_N)_{kl}) \E\left[\frac{\d^2 R_{ij}(z) X_{jl} R_{ik}(\overline{z})}{\d X_{kl}^2} \right]\\
=&\frac{1}{2N^{3/2} }  \sum_{l=1}^n  \kappa_3((A_N)_{kl}) \E\left[ \frac{\d^2 (R_N(z)X)_{il} }{ \d^2 X_{kl} }R_{ik}(\overline{z}) + 2\frac{\d (R_N(z) X)_{il}}{\d X_{kl} } \frac{\d R_{ik}(\overline{z}) }{\d X_{kl}}+(R_N(z)X)_{il} \frac{\d^2 R_{ik}(\overline{z})}{\d^2X_{kl}} \right] \\
 \end{split}
\end{equation}
The first subsum is:
\begin{equation}
\begin{split}
& \frac{1}{2N^{3/2} }  \sum_{l=1}^n  \kappa_3((A_N)_{kl})  \left( R_{ik}(z) (X^* R_N(z)X)_{ll} (R_N(z) X_{kl}) R_{ik}(\overline{z}) \right.\\
&~~~+ 2(R(z)X)_{il} (R_N(z)X)_{kl} (R_N(z)X)_{kl} R_{ik}(\overline{z}) + 2R_{ik}(z) (R_N(z) X)_{kl}  R_{ik}(\overline{z})\\
&~~~ +2 R_{ik}(z) (X^* R_N(z))_{lk} (X^* R_N(z)X)_{ll} R_{ik}(\overline{z}) \\
&~~~+\left. 2(R_N(z)X)_{il} R_{kk}(z)(X^*R_N(z)X)_{ll} R_{ik}(\overline{z}) \right)
\end{split}
\end{equation}
The second subsum is:
\begin{equation}
\begin{split}
&\frac{1}{2N^{3/2} }  \sum_{l=1}^n  \kappa_3((A_N)_{kl}) \left(2 R_{ik}(z) (X^* R_N(z) X)_{ll} R_{ik}(\overline{z}) (X^* R_N(\overline{z}))_{lk}  \right.\\
&~~~  +2 (R_N(z)X)_{il} (R_N(z)X)_{kl} R_{ik}(\overline{z}) (X^* R_N(\overline{z}))_{lk} \\
&~~~+2R_{ik}(z) (X^* R_N(z) X)_{ll} (R_N(\overline{z})X)_{il} R_{kk}(\overline{z})  +     2    (R_N(z) X)_{il}(z) (R_N(z)X)_{kl}(R_N(\overline{z})X)_{il} R_{kk}(\overline{z})\\
&~~~+\left. 2  R_{ik}(z) R_{ik}(\overline{z}) (X^* R_N(\overline{z}))_{lk} + 2R_{ik}(z)(R_N(\overline{z})X)_{il} R_{kk}(\overline{z}) \right)
\end{split}
\end{equation}
The third subsum is:
\begin{equation}
\begin{split}
&\frac{1}{2N^{3/2} }  \sum_{l=1}^n   \kappa_3((A_N)_{kl})  \left(2  (R_N(z) X)_{il} R_{ik}(\overline{z})  (X^* R_N(\overline{z}) X)_{ll} R_{kk}(\overline{z})  \right.\\
&~~~  + 2 (R_N(z) X)_{il}(R_N(\overline{z}) X)_{il} (R_N(\overline{z}) X)_{kl} R_{kk}(\overline{z})  \\
&~~~ + 2  (R_N(z) X)_{il} R_{ik}(\overline{z})  R_{kk}(\overline{z})   +2  (R_N(z) X)_{il} R_{ik}(\overline{z})  (X^* R_N(\overline{z}) )_{lk} (X^* R_N(\overline{z}) )_{lk} \\
&~~~+ \left. 2 (R_N(z) X)_{il}(R_N(\overline{z}) X)_{il} R_{kk}(\overline{z}) (X^*R_N(\overline{z}) )_{lk} \right)
 \end{split}
\end{equation}
Once again by \eqref{ineq} and the Cauchy-Schwarz inequality this term is bounded by 
$$O\left(\frac{P_{2}(|\Im(z)|^{-1}) \E[ P_4(\|X_N\|) \* \|R_N(z)\|^2]}{N}\right).$$
The error term due to the truncation of the decoupling formula at $p=2$ is estimated as in (\ref{dinamo}-\ref{dinamo2}) in Lemma \ref{error1}.
\end{proof}
Then using  \eqref{master1} to subtract $\E[R_{ik}(z)] \E[R_{ik}(\overline{z})]$ from \eqref{mastervar} gives:
\begin{align}
&(z -\sigma^2z g_N(z) + \sigma^2(1 - c_N)  )( \E[R_{ik}(z)R_{ik}(\overline{z})] - \E[R_{ik}(z)] \E[R_{ik}(\overline{z})])\\
 &=  O\left( \frac{P_{3}(|\Im(z)|^{-1}) (\E[P_{10}(\|(X_N\|) \* (\|R_N(z)\|^2 + \|R_N(z)\|^{3/2})}{N} \right) 
\end{align}
Repeating the argument for \eqref{expik} leads to:
$$\V(R_{ik}(z)) = O\left( \frac{P_{4}(|\Im(z)|^{-1}) (\E[P_{10}(\|(X_N\|) \* (\|R_N(z)\|^2 + \|R_N(z)\|^{3/2})}{N} \right).$$
\end{proof}

\section{ Functional Calculus }  \label{sect:functional}
\label{sec:proofs}
We now extend the results of Section \ref{sec:estvar} from resolvents to a more general class of functions. To do this we use the 
Helffer-Sj\"ostrand functional calculus (\cite{HS}, \cite{D}). Let $f \in C^{l+1}(\R)$, functions with $l+1$ continuous derivatives that decay at infinity sufficiently fast. Then one can write
\begin{equation}
 f(X_N)=-\frac{1}{\pi}\,\int_{\mathbb{C}}\frac{\partial \tilde{f}}{\partial \bar{z}}\, R_N(z)\,dxdy 
\quad,\quad\frac{\partial \tilde{f}}{\partial \bar{z}} := \frac{1}{2}\Big(\frac{\partial \tilde{f}}
{\partial x}+i\frac{\partial \tilde{f}}{\partial y}\Big)
\label{formula-H/S}
 \end{equation}
 where:
 \begin{itemize}
\item[i)]
 $z=x+iy$ with $x,y \in \mathbb{R}$;
 \item[ii)] $\tilde{f}(z)$ is the extension of the function $f$ defined as follows
  \begin{equation}\label{a.a. -extension}
  \tilde{f}(z):=\Big(\,\sum_{n=0}^{l}\frac{f^{(n)}(x)(iy)^n}{n!}\,\Big)\sigma(y);
\end{equation}
 here $\sigma \in C^{\infty}(\mathbb{R})$ is a nonnegative function equal to $1$ for $|y|\leq 1/2$ and equal to zero for $|y|\geq 1$.
 \end{itemize}
From its definition one can see that \eqref{a.a. -extension} satisfies the following bound: \begin{equation}\label{estimate-derivative3}
\Big|\frac{\partial \tilde{f}}{\partial \bar{z}} (x+iy)\Big|\leq  Const \* \max\left(|\frac{d^jf}{dx^j}(x)|, \ 0\leq j \leq l+1\right) \*
|y|^l\quad.
\end{equation}

\begin{proposition}
\label{expandvarf} Let $A_N$ be an $N \times n$ real (complex) matrix that satisfies condition \textbf{C2} (\textbf{C1}). 
Let $M_N = \frac{1}{N}A_N A_N^*$.
\begin{enumerate}[(i)]
\item Let $f\colon \R \to \R $ such that $supp(f) \cap \R_+ \subset [0,L]$ for some $L>0$ and \\
$\|f\|_{C^7([0,L])} < \infty$, 
then there exists a constant, 
$C(L,\sigma, m_4)$, such that:
\begin{equation}
\label{expiif}
\left| \E[f(M_N)_{ii}] - \int f(x) d \mu_{\sigma,c_N}(x)\right| \leq C(L,\sigma, m_4) \|f\|_{C^7([0,L])} N^{-1} \text{ for } 1 \leq i \leq N.
\end{equation}
\item Let $f\colon \R \to \R $ such that $\|f\|_{C^6(\R_+)} < \infty$, then there exists a constant, $C(\sigma, m_4)$, such that:
\begin{equation}
\label{expikf}
\left| \E[f(M_N)_{ik}] \right| \leq C(\sigma, m_4) \|f\|_{C^6(\R_+)} N^{-1} \text{ for } 1 \leq i \not= k \leq N.
\end{equation}
\item Let $f\colon \R \to \R $ such $\|f\|_{s} < \infty$, for $s>3$ then there exists a constant, $C(s, \sigma, m_4)$ such that:
\begin{equation}
\label{varijf}
\left| \V[f(M_N)_{ij}] \right| \leq C(s, \sigma, m_4) \|f\|_{s} N^{-1}   \text{ for } 1 \leq i , j \leq N,
\end{equation}
\end{enumerate}
\end{proposition}
The proof follows as in \cite{ors}. We sketch the details below.
\begin{proof}
First, we note that since $M_N$ is a non-negative definite matrix, changing the values of $f(x)$ for negative $x$ does not have any effect on
the matrix values $f(M_N)_{ij}.$   For example, we can always multiply $f$ by a smooth function $\varphi$ which is identically $1$ on $R_+$ and $0$ on
$(-\infty, -\delta].$ By the Helffer-Sj\"ostrand functional calculus we have:
\begin{equation}
\begin{split}
 \E[f(M_N)_{ii}] &=  \E\left[-\frac{1}{\pi}\,\int_{\mathbb{C}}\frac{\partial \tilde{f}}{\partial \bar{z}}\, R_{ii}(z)\,dxdy\right] \\
&=  \E\left[-\frac{1}{\pi}\,\int_{\mathbb{C}}\frac{\partial \tilde{f}}{\partial \bar{z}}\, (g_\sigma(z) + \epsilon_{ii}) \,dxdy\right] \\
&=  \int f(x)  d\mu_{\sigma,c_N}(x) -  \E\left[\frac{1}{\pi}\,\int_{\mathbb{C}}\frac{\partial \tilde{f}}{\partial \bar{z}} \epsilon_{ii} \,dxdy\right] 
\end{split}
\end{equation}
Where 
\[ |\epsilon_{ii}| = |\E[ R_{ii}(z) - g_{\sigma,c_N}(z) | \leq \left|\frac{P_{6}(|\Im(z)|^{-1})}{N} \right| \]
by \eqref{expii}. 
Combining this inequality with \eqref{estimate-derivative3}, letting $l=6$ yields:
\begin{equation}
\begin{split}
 \E\left[\frac{1}{\pi}\,\int_{\mathbb{C}} \frac{\partial \tilde{f}}{\partial \bar{z}} \epsilon_{ii} \,dxdy\right] \leq C  \|f\|_{C^7([0,L])}N^{-1} 
\end{split}
\end{equation}
Completing the proof of \eqref{expiif}. The proof of \eqref{expikf} follows similarly.
%


%

The rest of the proof of \eqref{varijf} follows the proof of Proposition 4.2 in \cite{ors}, using Proposition 1 from \cite{Sh}.

We first consider the diagonal case, $i=j$, without loss of generality let $i=1$ and define the random spectral measure
\begin{equation}
\label{specmeas}
 \mu(dx,\omega) := \sum_{l=1}^N \delta(x - \lambda_l) |\phi_{l}(1)|^2 
 \end{equation}
Where $\lambda_l$ are the eigenvalues of $M_N$ and $\phi_l$ are the corresponding normalized eigenvectors. 

Proposition $2.2$ of \cite{ors} applied to the measure \eqref{specmeas} gives

\[
\V[f(M_N)_{11}] \leq C_s \|f\|_s^2 \int_{0}^{\infty} dy e^{-y} y^{2s-1} \int_{-\infty}^{\infty} \V[R_{11}(x+iy)] dx .\]

Using \eqref{varij} we can estimate $\int_{-\infty}^{\infty} \V[R_{ij}(x+iy)] dx$ from above by
\begin{align}
 \frac{P_4(y^{-1})}{N} \E\left[\int_{-\infty}^{\infty} P_{10}(\|X_N\|) \*(\|R_N(x+iy)\|^2 +\|R_N(x+iy)\|^{3/2}) dx\right] 
 \end{align}
Once we open the brackets, we obtain two terms. Here, we bound the first term. The other term can be estimated in a similar way. 
\begin{align}
\label{tor1}
& \E\left[\int_{-\infty}^{\infty} P_{10}(\|X_N\|) \|R_N(x+iy)\|^2 \*  dx  \right] \leq\\
\label{tor2}
 & \E \left[  P_{10}(\|X_N\|) \* 
\left(\int_{-\|X_N X_N^*\|}^{\|X_N X_N^*\|} y^{-2} dx + \int_{|x|>\|X_N X_N^*\|} \frac{ 1}{(x-\|X_NX_N^*\|)^2 + y^2} dx \right) \right]  \\
\label{tor3}
& \leq P_2(y^{-1}) \* \E P_{12}(\|X_N\|).
 \end{align}
By Lemma \ref{lemma:expecation-norm} (\ref{tor3}) can be bounded by $C \* P_2(y^{-1})$.
This leads to 
\[ \V[f(X_N)_{11}] \leq C_s \frac{\|f\|_s^2}{N} \int_{0}^\infty dy e^{-y} y^{2s-1} P_6(y^{-1}).\]
The integral converges if $s>3$.\\
In the off-diagonal case $i \not=j$, we consider the (complex-valued) measure
\[ \mu(dx,\omega):= \sum_{l=1}^N \delta(x - \lambda_l) \overline{\phi_l(i)} \phi_l(j) ,\]
which is a linear combination of probability measures, and apply Proposition 2.2 of \cite{ors} to each probability measure in the linear combination. 
Proposition \ref{expandvarf} is proven. 
\end{proof}

\section{Resolvent CLT} \label{sect:resolvent}

Let $m$ be a fixed positive integer and let $R_N^{(m)}(z)$ denote the $m \times m$ upper-left corner of the resolvent matrix, $R_N(z)$.  Define 
\begin{equation*}
	\Psi_N(z) = \sqrt{N} \left(R_N^{(m)}(z) - g_{\sigma, c_N}(z)I \right), z \in \C \setminus [0, \sigma^2(1+\sqrt{c})^2].
\end{equation*}
Clearly, $\Psi_N(z)$ is well defined for $z \in \C \setminus \R.$   By Lemma \ref{lemma:spectral-norm}, $\Psi_N(z)$ is well defined for 
$z \in \R \setminus [0, \sigma^2(1+\sqrt{c})^2]$ with probability going to $1.$
We are interested in studying the random function $\Psi_N(z)$ whose values are in the space of complex symmetric $m \times m$ matrices.  We also define
\begin{equation*}
	\varphi(z,w) = \E\left[ \frac{z}{z-\eta_{1/c}} \frac{w}{w-\eta_{1/c}} \right]
\end{equation*}
where $\eta_{1/c}$ is a Marchenko-Pastur distributed random variable with ratio index $\frac{1}{c}$ and scale index $\sigma^2$.  We introduce the following notation
\begin{align*}
	\varphi_{++}(z,w) &= \E\left[ \Re \frac{z}{z-\eta_{1/c}} \Re \frac{w}{w-\eta_{1/c}} \right] \\
	\varphi_{--}(z,w) &= \E\left[ \Im \frac{z}{z-\eta_{1/c}} \Im \frac{w}{w-\eta_{1/c}} \right] \\ 
	\varphi_{+-}(z,w) &= \E\left[ \Re \frac{z}{z-\eta_{1/c}} \Im \frac{w}{w-\eta_{1/c}} \right].
\end{align*}

\begin{theorem} \label{thm:resolvent-clt}
Let $A_N$ be a $N \times n$ random matrix with real entries that satisifies condition {\bf C2}.  Let $m$ be a fixed positive integer and assume that for 
$1 \leq i \leq m$
\begin{equation} \label{eq:k_4}
	m_4(i) := \lim_{N \rightarrow \infty} \frac{1}{n} \sum_{j} \E|A_{ij}|^4 
\end{equation}
exists and for all $\epsilon>0$ \eqref{eq:lf0.25} holds.  Let
\begin{equation*}
	\kappa_4(i) := m_4(i) - 3 \sigma^4, \ 1 \leq i \leq m.
\end{equation*}
Then the random field $\Psi_N(z)$ converges in finite-dimensional distributions to the random field $\Psi(z) = \sqrt{c} g_{\sigma,c}^2(z) Y(z)$ where
\begin{equation*}
	Y(z) = \left( Y_{ij}(z) \right)_{1 \leq i,j \leq m}
\end{equation*}
is the Gaussian random field such that the entries $Y_{ij}(z)$, $i\leq j$ and $Y_{kl}(w)$, $k \leq l$ are independent when $(i,j) \neq (k,l)$ and 
\begin{align*}
	\cov( \Re Y_{ii}(cz), \Re Y_{ii}(cw) ) &= \kappa_4(i) \Re[z g_{\sigma, 1/c}\left( z \right)] \Re [w g_{\sigma, 1/c}(w)] + 2 \sigma^4 \varphi_{++}\left(z,w\right), \\
	\cov( \Im Y_{ii}(cz), \Im Y_{ii}(cw) ) &= \kappa_4(i) \Im[z g_{\sigma, 1/c}\left( z \right)] \Im [w g_{\sigma, 1/c}\left(w \right)] + 2 \sigma^4 \varphi_{--}\left(z,w\right), \\
	\cov( \Re Y_{ii}(cz), \Im Y_{ii}(cw) ) &= \kappa_4(i) \Re [z g_{\sigma, 1/c}\left( z \right)] \Im [w g_{\sigma, 1/c}\left(w\right)] + 2 \sigma^4 \varphi_{+-}\left(z,w\right), \\
	\cov( \Re Y_{ij}(cz), \Re Y_{ij}(cw) ) &= \sigma^4 \varphi_{++}\left(z,w\right) , i \neq j, \\
	\cov( \Im Y_{ij}(cz), \Im Y_{ij}(cw) ) &= \sigma^4 \varphi_{--}\left(z,w\right) , i \neq j, \\
	\cov( \Re Y_{ij}(cz), \Im Y_{ij}(cw) ) &= \sigma^4 \varphi_{+-}\left(z,w\right) , i \neq j. 
\end{align*}
\end{theorem}

In the Hermitian case we have the following.

\begin{theorem} \label{thm:resolvent-clt-complex}
Let $A_N$ be a $N \times n$ random matrix with complex entries that satisifies condition {\bf C1}.  Let $m$ be a fixed positive integer and assume that for $1 \leq i \leq m$
\begin{equation}
	m_4(i) := \lim_{N \rightarrow \infty} \frac{1}{n} \sum_{j} \E|A_{ij}|^4 
\end{equation}
exists and for all $\epsilon>0$ \eqref{eq:lf0.25} holds.  Let
\begin{equation*}
	\kappa_4(i) := m_4(i) - 2 \sigma^4, \ 1 \leq i \leq m.
\end{equation*}
Then the random field $\Psi_N(z)$ converges in finite-dimensional distributions to the random field $\Psi(z) = \sqrt{c} g_{\sigma,c}^2(z) Y(z)$ where
\begin{equation*}
	Y(z) = \left( Y_{ij}(z) \right)_{1 \leq i,j \leq m}
\end{equation*}
is the Gaussian random field such that the entries $Y_{ij}(z)$, $i\leq j$ and $Y_{kl}(w)$, $k \leq l$ are independent when $(i,j) \neq (k,l)$ and 
\begin{align*}
	\cov( \Re Y_{ii}(cz), \Re Y_{ii}(cw) ) &= \kappa_4(i) \Re[z g_{\sigma, 1/c}\left( z \right)] \Re [w g_{\sigma, 1/c}(w)] + \sigma^4 \varphi_{++}\left(z,w\right), \\
	\cov( \Im Y_{ii}(cz), \Im Y_{ii}(cw) ) &= \kappa_4(i) \Im[z g_{\sigma, 1/c}\left( z \right)] \Im [w g_{\sigma, 1/c}\left(w \right)] + \sigma^4 \varphi_{--}\left(z,w\right), \\
	\cov( \Re Y_{ii}(cz), \Im Y_{ii}(cw) ) &= \kappa_4(i) \Re [z g_{\sigma, 1/c}\left( z \right)] \Im [w g_{\sigma, 1/c}\left(w\right)] + \sigma^4 \varphi_{+-}\left(z,w\right), \\
	\cov( \Re Y_{ij}(cz), \Re Y_{ij}(cw) ) &= \frac{1}{2} \sigma^4 \left( \varphi_{++}\left(z,w\right) + \varphi_{--}(z,w) \right) , i \neq j, \\
	\cov( \Im Y_{ij}(cz), \Im Y_{ij}(cw) ) &= \frac{1}{2} \sigma^4 \left( \varphi_{++}\left(z,w\right) + \varphi_{--}(z,w) \right) , i \neq j, \\
	\cov( \Re Y_{ij}(cz), \Im Y_{ij}(cw) ) &= \frac{1}{2} \sigma^4 \left( \varphi_{+-}\left(z,w\right) - \varphi_{+-}(w,z) \right) , i \neq j. 
\end{align*}
\end{theorem}

\begin{remark}
We remind the reader that the covariance values in Theorems \ref{thm:resolvent-clt} and \ref{thm:resolvent-clt-complex} are stated in terms of the Marchenko-Pastur law with ratio index $\frac{1}{c}$ and scale index $\sigma^2$.  In some cases it may be more convenient to state the covariances in terms of the Marchenko-Pastur law with ratio index $c$.  Indeed, a simple computation reveals that for $c>0$ and any continuous function $f$, 
\begin{equation*}
	\E \left[ f(\eta_{1/c}) \right] = \frac{1}{c} \E \left[ f \left( \frac{\eta_c }{c} \right) \right] + \left(1 - \frac{1}{c} \right) f(0)
\end{equation*}
where $\eta_c$ is a Marchenko-Pastur distributed random variable with ratio index $c$ and scale index $\sigma^2$.  In particular, we note that 
\begin{equation}
\label{grelation}
	g_{\sigma, 1/c}(z) = g_{\sigma, c}(cz) + \left( 1 - \frac{1}{c} \right) \frac{1}{z}, 
\end{equation}
\begin{equation}
\label{phirelation}
	\varphi(z,w) = \frac{1}{c} \E \left[ \frac{cz}{cz - \eta_c} \frac{cw}{cw-\eta_c} \right] + \left( 1 - \frac{1}{c} \right).
\end{equation}
\end{remark}

We will need the following lemma for the proof of Theorems \ref{thm:resolvent-clt} and \ref{thm:resolvent-clt-complex}.  

\begin{lemma} \label{lemma:commute}
Let $B$ be an $N \times n$ matrix.  Then
\begin{equation*}
	B^\ast (z-BB^\ast)^{-1}B = B^\ast B(z- B^\ast B)^{-1}
\end{equation*}
for all $z \notin \text{Sp}(BB^\ast)\cup\{0\}$.  
\end{lemma}

\begin{proof}
Choose $z \notin \text{Sp}(BB^\ast)\cup\{0\}$ such that $|z| > \|BB^\ast\|$.  Then we have that
\begin{align*}
	B^\ast (z-BB^\ast)^{-1}B &= B^\ast \frac{1}{z} \left( I + \sum_{k=1}^\infty \frac{1}{z^k} (B B^\ast)^k \right) B \\
		&= B^\ast B \frac{1}{z} \left(I + \sum_{i=1}^\infty \frac{1}{z^k} (B^\ast B)^k \right) \\
		&= B^\ast B(z- B^\ast B)^{-1}.
\end{align*}
We can now extend the result to all $z \notin \text{Sp}(BB^\ast)\cup\{0\}$ by analytic continuation of the function 
\begin{equation*}
	f_{uv}(z) = \langle B^\ast (z-BB^\ast)^{-1}B u, v \rangle
\end{equation*}
where $u,v$ are arbitrary vectors.  
\end{proof}

We present the proof of Theorem \ref{thm:resolvent-clt} below.  The proof in the Hermitian case is similar (see also 
\cite{PRS} and \cite{ors}) and is left to the reader.  

\begin{proof}[Proof of Theorem \ref{thm:resolvent-clt}]

We write
\begin{equation*}
	A_N = \left( \begin{array}{c}
		r_1  \\
		r_2 \\
		\vdots \\
		r_N \end{array} \right), \qquad 
	A_N^{(m)} = \left( \begin{array}{c}
		r_{m+1}  \\
		r_{m+2} \\
		\vdots \\
		r_N \end{array} \right),
\end{equation*}
where $r_i$ is an $n$-vector representing the $i$-th row of $A_N$.  We remind the reader that $X_N = \frac{1}{\sqrt{N}} A_N$ and we will use the 
notation $X_N^{(m)} = \frac{1}{\sqrt{N}} A^{(m)}_N$.  Recall that we denote by $R_N^{(m)}(z)$ the $m \times m$ upper-left corner of the resolvent matrix, 
$R_N(z),$  of $M_N=\frac{1}{N}\*A_N\*A_N^*=X_N\*X_N^*.$

Standard linear algebra gives (see e.g. \cite{ors})
\begin{equation*}
	R_N^{(m)}(z) = \left( z I_m - \left( \frac{1}{N} r_i B_{N}^{(m)}(z) r_j^\ast \right)_{i,j=1}^m \right)^{-1}
\end{equation*}
where 
\begin{equation} \label{eq:B-def}
	B_N^{(m)}(z) = I_n + {X_N^{(m)}}^\ast \left(z-X_N^{(m)} {X_N^{(m)}}^\ast \right)^{-1} X_N^{(m)}.
\end{equation}

Let
\begin{equation*}
	\Gamma_N(z) = \sqrt{N} \left[ \left( \frac{1}{N} r_i B_{N}^{(m)}(z) r_j^\ast \right)_{i,j=1}^m - \sigma^2( c_N -1 + z g_{\sigma, c_N})I_m \right].
\end{equation*}
Then a simple computation reveals that
\begin{equation*}
	R_N^{(m)} = \left[ \frac{1}{g_{\sigma,c_N}(z)} I_m - \frac{1}{\sqrt{N}} \Gamma_N(z) \right]^{-1}.
\end{equation*}

It will follow from the Central Limit Theorem for Quadratic forms (see the appendix of \cite{ors}), that $\| \Gamma_N(z) \|$ is bounded in probability for $z \in \C \setminus \R$.  Thus, we have that
\begin{equation*}
	\Psi_N(z) = g^2_{\sigma, c_N}(z) \Gamma_N(z) + o(1).
\end{equation*}

We note that by \eqref{eq:B-def},
\begin{equation*}
	\frac{1}{N} \Tr B_N^{(m)}(z) = c_N - 1 + \frac{z}{N} \Tr \left(z - X_N^{(m)} {X_N^{(m)}}^\ast \right)^{-1}
\end{equation*}
and
\begin{equation*}
	\sqrt{N} \left[ \frac{1}{N} \Tr (z - X_N^{(m)} {X_N^{(m)}}^\ast)^{-1} - g_{\sigma, c}(z) \right] \longrightarrow 0
\end{equation*}
in probability as $N \rightarrow \infty$ by \eqref{tracevarbound}.  

Thus, we have that
\begin{align*}
	\Gamma_N(z) &= \sqrt{N} \left[ \left( \frac{1}{N} r_i B_{N}^{(m)}(z) r_j^\ast \right)_{i,j=1}^m - \frac{\sigma^2}{N} \Tr B_N^{(m)}(z) I_m \right] + o(1) \\
		&= \sqrt{c_N n} \left[ \left( \frac{1}{n} r_i B_{N}^{(m)}(z) r_j^\ast \right)_{i,j=1}^m - \frac{\sigma^2}{n} \Tr B_N^{(m)}(z) I_m \right] + o(1).  
\end{align*}
By Lemma \ref{lemma:commute} and \eqref{eq:B-def}, we have that
\begin{equation} \label{eq:commute}
	B_N^{(m)}(z) = z(z- {X_N^{(m)}}^\ast X_N^{(m)})^{-1} = \frac{z}{c_N} \left( \frac{z}{c_N} - \frac{1}{n} {A_N^{(m)}}^\ast A_N^{(m)} \right)^{-1}.
\end{equation}

By Theorem \ref{thm:mp} and \eqref{eq:commute}, we note that
\begin{align*}
	\frac{1}{n} \Tr \left[ \Re B_N^{(m)}(z) \Re B_N^{(m)}(w) \right] &\longrightarrow \varphi_{++}\left( \frac{z}{c}, \frac{w}{c} \right),\\
	\frac{1}{n} \Tr \left[ \Im B_N^{(m)}(z) \Im B_N^{(m)}(w) \right] &\longrightarrow \varphi_{--}\left( \frac{z}{c}, \frac{w}{c} \right),\\
	\frac{1}{n} \Tr \left[ \Re B_N^{(m)}(z) \Im B_N^{(m)}(w) \right] &\longrightarrow \varphi_{+-}\left( \frac{z}{c}, \frac{w}{c} \right),
\end{align*}
in probability as $N \to \infty$.  We now claim that for $1 \leq i \leq m$,
\begin{align*}
	\frac{1}{n} \sum_{j=1}^n \kappa_4(A_N)_{ij}\left[ \Re B_N^{(m)}(z) \right]_{jj} \left[ \Re B_N^{(m)}(w) \right]_{jj} &\longrightarrow \kappa_4(i) \Re\left[ \frac{z}{c} g_{\sigma, 1/c}\left(\frac{z}{c} \right) \right] \Re\left[ \frac{w}{c} g_{\sigma, 1/c}\left(\frac{w}{c} \right) \right], \\
	\frac{1}{n} \sum_{j=1}^n \kappa_4(A_N)_{ij}\left[ \Im B_N^{(m)}(z) \right]_{jj} \left[ \Im B_N^{(m)}(w) \right]_{jj} &\longrightarrow \kappa_4(i) \Im\left[ \frac{z}{c} g_{\sigma, 1/c}\left(\frac{z}{c} \right) \right] \Im\left[ \frac{w}{c} g_{\sigma, 1/c}\left(\frac{w}{c} \right) \right], \\
	\frac{1}{n} \sum_{j=1}^n \kappa_4(A_N)_{ij}\left[ \Re B_N^{(m)}(z) \right]_{jj} \left[ \Im B_N^{(m)}(w) \right]_{jj} &\longrightarrow \kappa_4(i) \Re\left[ \frac{z}{c} g_{\sigma, 1/c}\left(\frac{z}{c} \right) \right] \Im\left[ \frac{w}{c} g_{\sigma, 1/c}\left(\frac{w}{c} \right) \right], \\
\end{align*}
in probability as $N \to \infty$.  Indeed, for the first statement, by the triangle inequality
\begin{align*}
	&\left| \left[ \Re B_N^{(m)}(z) \right]_{jj} \left[ \Re B_N^{(m)}(w) \right]_{jj} - \Re\left[ \frac{z}{c} g_{\sigma, 1/c}\left(\frac{z}{c} \right) \right] \Re\left[ \frac{w}{c} g_{\sigma, 1/c}\left(\frac{w}{c} \right) \right] \right| \\
	& \qquad \leq \frac{ |w|}{|\Im w|} \left| \left[ \Re B_N^{(m)}(z) \right]_{jj} - \Re\left[ \frac{z}{c} g_{\sigma, 1/c}\left(\frac{z}{c} \right) \right] \right| \\
	& \qquad \qquad + \frac{|z|}{|\Im z|} \left| \left[ \Re B_N^{(m)}(w) \right]_{jj} - \Re\left[ \frac{w}{c} g_{\sigma, 1/c}\left(\frac{w}{c} \right) \right] \right|.
\end{align*}
By Proposition \ref{expandvar}, we obtain
\begin{equation*}
	 \left[ \Re B_N^{(m)}(z) \right]_{jj} \left[ \Re B_N^{(m)}(w) \right]_{jj} = \Re\left[ \frac{z}{c} g_{\sigma, 1/c}\left(\frac{z}{c} \right) \right] \Re\left[ \frac{w}{c} g_{\sigma, 1/c}\left(\frac{w}{c} \right) \right] + o(1).  
\end{equation*}
The claim is then complete by assumption \eqref{eq:k_4}.  The other two statements follow from the same argument.  

Fix $p \geq 1$ and consider $z_1, \ldots, z_p \in \C \setminus \R$.  We define the family of matrices
\begin{equation*}
	C_N^{s,t} = \sum_{l=1}^p \left[ a_{s,t}^{(l)} \Re B_N^{(m)}(z_l) + b_{s,t}^{(l)} \Im B_N^{(m)}(z_l) \right], 1 \leq s, t \leq m
\end{equation*}
where $a_{s,t}^{(l)}$ and $b_{s,t}^{(l)}$ are arbitrary real constants for $1 \leq s \leq t \leq m, 1 \leq l \leq p$.  We now apply the Central Limit Theorem for Quadratic forms (see the appendix of \cite{ors}) to the family of matrices $C_N^{s,t}$ and use the above computations to conclude that $\Gamma_N(z)$ converges in finite dimensional distributions to $\sqrt{c} Y(z)$ for $\Im z \neq 0$.  

For $z \in \R \setminus [0, \sigma^2 (1 + \sqrt{c})^2]$, define $\delta = \frac{1}{3} \text{dist}(z, [0, \sigma^2 (1 + \sqrt{c})^2])$.  Let $h(x)$ be a smooth function with compact support where 
\begin{align*}
	h(x)&=0  \text{ for } x \notin [-2 \delta, \sigma^2 (1+\sqrt{c})^2 + 2 \delta],\\
	h(x)&=1 \text{ for } x \in [-\delta, \sigma^2 (1+\sqrt{c})^2 + \delta].
\end{align*}
To complete the proof for $z \in \R \setminus [0, \sigma^2 (1 + \sqrt{c})^2]$, we repeat the same arguments as above replacing $B_N^{(m)}$ with $h(X_N X_N^\ast) B_N^{(m)}$, where $h$ is a $C^\infty(\R)$ function with compact support such that 
\begin{equation}
\label{h}
	h(x) := 1 \text{ for } x \in [-\delta, \sigma^2(1 + \sqrt{c})^2 + \delta],
\end{equation}
for some $\delta>0$.  
It is essential here that
\begin{equation*}
	\Prob ( B_N^{(m)} \neq h(X_N X_N^\ast) B_N^{(m)} ) \longrightarrow 0
\end{equation*}
as $N \rightarrow \infty$ by Lemma \ref{lemma:spectral-norm}.  

\end{proof}

\section{Fluctuations of matrix entires for regular functions}

We now prove Theorem \ref{thm:main_real}, Theorem \ref{thm:main_complex} follows similarly.

\begin{proof}

In Theorem \ref{thm:resolvent-clt}, Theorem \ref{thm:main_real} is proved for functions of the form 
\begin{equation}
\label{lincom}
 \sum_{l=1}^k a_l \frac{1}{z_{l}-x} , ~~ z_l \not \in [0, \sigma^2(1+\sqrt{c})^2] ,~~ 1\leq l\leq k .
 \end{equation}
It follows from \eqref{utro11} and \eqref{phirelation} that the limiting variance for functions of the form 
(\ref{lincom}) given in Theorem \ref{thm:resolvent-clt} coincides with the one given in  Theorem \ref{thm:main_real}.
We recall that $\sqrt{N}\* (R_N(z)_{ij} -\delta_{ij} \*g_{\sigma, c_N}(z)), \ 1\leq i,j\leq m, $ converges in finite-dimensional distributions
to the random point field $\sqrt{c}\* g^2_{\sigma, c}(z)\* Y_{ij}(z), \ 1\leq i,j\leq m.$
For the off-diagonal entries $i\not=j$, one has
\begin{equation}
\begin{split}
&g_{\sigma,c}^2(z) g_{\sigma,c}^2(w) Cov(Y_{ij}(z),Y_{ij}(w)) = g_{\sigma,c}^2(z) g_{\sigma,c}^2(w) \sigma^4 \varphi \left(\frac{z}{c},\frac{w}{c} \right)\\
&= -\sigma^4g_{\sigma,c}^2(z) g_{\sigma,c}^2(w) \frac{zw(g_{\sigma,c}(z)- g_{\sigma,c}(w))}{c(z-w)} - \sigma^4 \frac{1-c}{c}  g_{\sigma,c}^2(z) g_{\sigma,c}^2(w)\\
 &=-\sigma^4g_{\sigma,c}(z) g_{\sigma,c}(w)\frac{1}{c}\left( (c-1) g_{\sigma,c}(w)g_{\sigma,c}(z) +\frac{z g_{\sigma,c}(z) - w g_{\sigma,c} (w)  }{\sigma^2(z-w)} \right) \\
&=-\sigma^4\frac{1}{c\sigma^4}g_{\sigma,c}(z)g_{\sigma,c}(w)  -\sigma^4 \frac{ g_{\sigma,c}(z)-g_{\sigma,c}(w)   }{c\sigma^4(z-w)}\\
&=\frac{1}{c } Cov\left(\frac{1}{z-\eta},\frac{1}{w-\eta}  \right).
\end{split}
\end{equation}
The calculations in the diagonal case $i=j$ are similar.  To verify that the fourth cumulant term in
$ Cov(Y_{ii}(z),Y_{ii}(w))$ gives the required contribution \eqref{def:rho}, one uses the identity
\begin{equation}
\begin{split}
&\frac{z}{c} g_{\sigma,c}^2(z) g_{\sigma,1/c}\left(\frac{z}{c}\right) 
=\frac{1}{c\sigma^2} \left( zg_{\sigma,c}^2(z)- g_{\sigma,c} \right)       \\
&=\frac{1}{c\sigma^2} \left(   \frac{(z-\sigma^2 (c-1))g_{\sigma,c}(z)-1}{\sigma^2}  - g_{\sigma,c} \right) 
=\frac{1}{c\sigma^2} \* \frac{(z-c\*\sigma^2)g_{\sigma,c}(z) -1}{\sigma^2}. 
\end{split}
\end{equation}
Thus, Theorem \ref{thm:main_real} is proved for functions of the form 
\eqref{lincom}. By Lemma \ref{lemma:spectral-norm}, the result then also follows for test functions of the form
\begin{equation}
\label{lincom11}
 \sum_{l=1}^k a_l \* h_l(x) \*\frac{1}{z_{l}-x} , ~~ z_l \not \in [0, \sigma^2(1+\sqrt{c})^2] ,~~ 1\leq l\leq k,
 \end{equation}
where $h_l, \ 1\leq l \leq k, $ are $C^\infty(\R)$ functions with compact support that satisfy (\ref{h}).

 By the Stone-Weierstrass theorem (see e.g \cite{RS}), functions (\ref{lincom}) are dense in $C^p(X)$, for any compact set $X \subset \R$ and 
$p=0,1,2,\ldots.$
Let $f$ be a $C^\infty(\R)$ function such that $supp(f) \subset [-L,L]$ for some $L>0$. Then there exists a sequence of functions $f_j$ of the form
\eqref{lincom11} that converge to $f$ in the $C^4[-L, L]$ norm (for the definition of the $C^4[-L, L]$ norm see (\ref{Cknorm})). Then, $f_j$ also 
converges to $f$ in the $\mathcal{H}_s$ norm for $s\leq 4.$  
By the estimate \eqref{varijf}
\begin{equation}
\V[ \sqrt{N}(  f(X_N)_{ij} - f_k(X_N)_{ij} ) ] \leq const \* \|f(x) - f_k(x)\|_s.
\end{equation}
Since the r.h.s. can be made arbitrarily small,  and $\omega^2(f), \ \rho^2(f)$ in  (\ref{def:omega}), (\ref{def:rho}) are continuous in 
the $\mathcal{H}_s$ norm, we obtain that
$ \sqrt{N}(f(X_N)_{ij} - \E[f(X_N)_{ij}])$ converges in distribution to the Gaussian random variable defined in Theorem \ref{thm:main_real}. 
Since smooth functions with compact support are dense in $\mathcal{H}_s,$ the result can be extended to an arbitrary $f \in \mathcal{H}_s.$
\end{proof}

\end{document}